\documentclass{article}

\usepackage[utf8]{inputenc}
\usepackage[english]{babel}
\usepackage{amsfonts}
\usepackage{amssymb}
\usepackage{amsmath}
\usepackage[all]{xy}
\usepackage{mathrsfs}
\usepackage{latexsym}
\usepackage{amscd}

\usepackage{tikz-cd}

\usepackage[mathscr]{eucal}
\usepackage{amsthm}
\usepackage{setspace}                    % Line spacing
\usepackage{accents}
\newcommand{\dbtilde}[1]{\accentset{\approx}{#1}}

\usepackage{hyperref}
\newtheorem{prop}{Proposition}[section]
\newtheorem{proposition}[prop]{Proposition}

\newtheorem{theorem}[prop]{Theorem}

\newtheorem{lemma}[prop]{Lemma}

\newtheorem{corollary}[prop]{Corollary}
\newtheorem{conjecture}[prop]{Conjecture}

\theoremstyle{definition}

\newtheorem{definition}[prop]{Definition}

\theoremstyle{remark}

\newtheorem{remark}[prop]{Remark}

\newtheorem{question}[prop]{Question}

\newcommand{\Sym}{\mathrm{Sym}}

\newcommand{\el}{\ell}

\newcommand{\pp}{\mathbb{P}}
\newcommand{\PP}{\mathbb{P}}

\newcommand{\oo}{\mathcal{O}}

\newcommand{\GL}{\mathrm{GL}}

\newcommand{\rk}{\mathrm{rk}}

\newcommand{\CC}{\mathbb{C}}

\usepackage{tikz}
\usetikzlibrary{arrows}

\def\To{\longrightarrow}

\def\rmapdown#1{\Big\downarrow\rlap{$\vcenter{\hbox{$\scriptstyle
#1$}}$}}
\def\lmapdown#1{\Big\downarrow\llap{$\vcenter{\hbox{$\scriptstyle
#1\;\;\,$}}$}}

\def\squaremap#1#2#3#4#5#6#7#8{\mathop{
\begin{array}{ccccc}
#1 & \stackrel{#2}{\To}&  #3\\
\lmapdown{#4} & { } & \rmapdown{#5}\\
#6 & \stackrel{#7}{\To} & #8
\end{array}
}\nolimits}

\title{On semiample vector bundles and parallelizable compact complex manifolds}
\author{Francesco Esposito,
Ernesto C. Mistretta}

\date{}
\begin{document}

\maketitle

%%%% METTERE QUI L'EPIGRFE
%\epigraph{\itshape Think you're escaping and run into yourself. Longest way round is the shortest way home.}{---James Joyce, \textit{Uulysses - Nausicaa}} 
%

\begin{abstract}

We provide a characterization of parallelizable compact complex manifolds and their quotients using holomorphic symmetric differentials.
In particular we show that compact complex manifolds of Kodaira dimension 0 having 
strongly semiample cotangent bundle are parallelizable manifolds, while compact complex manifolds of Kodaira dimension 0 having 
weakly semiample cotangent bundle are quotients of parallelizable manifolds.
The main constructions used involve considerations about semiampleness of vector bundles, which are themselves of interest. As a byproduct we prove that compact manifolds having Kodaira dimension 0 and weakly sermiample cotangent bundle have infinite fundamental group, and we conjecture that this should be the case for all compact complex manifolds not of general type with weakly semiample cotangent bundle.

\end{abstract}

\section{Introduction}

In a previous work (\cite{mistrurbi}) the second named author with S. Urbinati showed that the semiampleness property for a vector bundle can be stated in two non-equivalent ways, 
both present in the literature, and they called the two notions weak semiampleness and strong semiampleness. Weak semiampleness of a vector bundle means that the tautological line bundle on the projectivization of the vector bundle is semiample, and is implied by strong semiampleness, which means that some symmetric power of the vector bundle is globally generated.

A geometrical interpretation of the difference between the two definitions can be found considering the cotangent bundle of surfaces:
Let $S$ be a smooth projective surface of Kodaira dimension $0$, then $S$ has strongly semiample cotangent bundle if and only if $S$ is an abelian surface, 
and $S$ has weakly semiample cotangent bundle if and only if $S$ is a hyperelliptic surface, 
\emph{i.e.} a quotient of an abelian surface (see paragraph \ref{main1} below for more details on this construction).

In the same work \cite{mistrurbi},  abelian varieties where characterized as those smooth projective varieties whose cotangent bundle has some symmetric product which is trivial, or equivalently have Kodaira dimension 0 and some symmetric product of the cotangent bundle is globally generated. Later, in \cite{mistrettaparma} the second named author proved that complex tori share the same property that characterizes them among compact K\"ahler manifolds.

It is worthwhile to note that a renewed interest has been shown recently  towards quotients of parallelizable manifolds,
in fact Catanese and Catanese-Corvaja (cf. \cite{catanesehyperI}, 
\cite{catanesehyperII}, \cite{catanesecorvaja}) 
give many interesting characterizations of those manifolds, disproving a conjecture by Baldassarri, and showing many of their geometrical properties, both in the K\"ahler and 
non K\"ahler case.

We prove here that the geometric interpretation stated above for projective surfaces
 can be given for all compact complex  manifolds:
compact complex parallelizable manifolds can be characterized
 as those compact complex manifolds of Kodaira dimension $0$ with strongly semiample cotangent bundle,
 and their quotients as 
 compact complex manifolds of Kodaira dimension $0$ with weakly semiample cotangent bundle.

This leads us to formulate some questions regarding a bimeromorphic characterization of parallelizable manifolds and of complex tori through generic generation properties of the symmetric powers of the cotangent bundle.

The characterization given for quotients of parallelizable manifolds relies on the following result, which is immediate in the case of smooth projective manifolds, due to intersection theory, but needs some work to be proven in general on compact complex manifolds:

\begin{theorem}
\label{mainsemiample}
Let $X$ be a compact complex manifold, and let $E$ be a weakly semiample holomorphic vector bundle on $X$. Then the determinant line bundle $\det (E)$ is semiample.
\end{theorem}
We will prove this theorem using the construction of the resultant as a global section of some power of the determinant line bundle:

\begin{lemma}
\label{lemmaresultintro}

Let $E$ be a rank $r$ holomorphic  vector bundle on a compact complex manifold $X$. 
Consider r sections $\sigma_1, \dots, \sigma_r$, with 
$\sigma_j \in H^0(\pp(E), \oo_{\pp(E)}(m_j))$, and $m_j>0$ positive integers.

Then there is a section 
$\mathrm{Res}(\sigma_1, \dots, \sigma_r) \in H^0(X, \det(E)^{\otimes m_1\cdot m_2 \dots m_r})$ 
vanishing exactly on the points $x\in X$ such that the sections
$ \sigma_{1 |\pp(E(x))} , \dots , \sigma_{r |\pp(E(x))}$ have a common zero in 
$\pp(E(x))$.
\end{lemma}

With the use of the theorem above, we will prove the following two characterization theorems:

\begin{theorem}
\label{parall}

Let $X$ be a compact complex manifold. Then the following are equivalent:

\begin{enumerate}

\item $X$ is a complex parallelizable manifold.

\item $X$ has Kodaira dimension 0, and the holomorphic cotangent bundle
$\Omega_X$ is \emph{strongly} semiample.

\item There is some positive integer $m>0$ such that $Sym^m \Omega_X$ is a trivial holomorphic vector bundle.

\end{enumerate}

\end{theorem}

\begin{theorem}
\label{hyperell}

Let $X$ be a compact complex manifold. Then the following are equivalent:

\begin{enumerate}

\item $X \cong \widetilde{X} / G$, where $\widetilde{X}$ is a complex parallelizable manifold, 
and $G$ is a finite group acting freely on $\widetilde{X}$.

\item $X$ has Kodaira dimension 0, and the holomorphic cotangent bundle
$\Omega_X$ is \emph{weakly} semiample (\emph{i.e.} there is some positive integer $m>0$ such that $\oo_{\PP(\Omega_X)}(m)$ is a base point free line bundle on 
$\PP(\Omega_X)$).

\end{enumerate}

\end{theorem}

Finally, we apply our constructions to prove the following:

\begin{theorem}

Let $X$ be a compact complex manifold of dimension $n$ and  Kodaira dimension
$k(X)$.

\begin{enumerate}

\item If $k(X) = 0$ and $\Omega^1_X$ is weakly semiample,
then the fundamental group of $X$ is infinite.

\item If $k(X) <n$, $X$ is projective, and $\Omega^1_X$ is weakly semiample,
then the fundamental group of $X$ is infinite.

\end{enumerate}

\end{theorem}

The structure of this work is the following:
in Section \ref{notation} we recall the basic definitions and previous results;
in Section \ref{lemmas} we prove the main results concerning semiample vector bundles, which we consider of interest in themselves;
in Section \ref{main1} we prove the main results concerning parallelizable manifolds and their quotients;
in Section \ref{conclusion} we obtain some information on the fundamental groups of manifolds with semiample cotangent bundle,  sketch the ideas on a corresponding bimeromorphic classification, and add further questions on the topics treated here.

\subsection{Acknowledgements}

We are deeply thankful for many interesting and pleasant conversations 
with our colleagues, especially with Y. Brunebarbe, F. Catanese, S. Diverio.

This research was partially funded by INDAM group GNSAGA,
BIRD research project BIRD201444 
``Varieties with low Kodaira dimension: Hyperkaehler manifolds, Fano manifolds, abelian varieties and parallelizable compact manifolds'',
PRIN research project 
``Curves, Ricci flat Varieties and their Interactions'', PRIN research project "Algebraic and geometric aspects of Lie theory".

\section{Notation and previous results}
\label{notation}

In order to construct properly the resultants associated to 
sections of the symmetric product of a vector bundle, 
we need to describe in a  detailed way the costruction
of the projectivization of a vector bundle.

\subsection{Notation for projective spaces and projectivization}
\label{projectivization}

Throughout this work we will
follow Grothendieck's notation for projective spaces, \emph{i.e.} 
$\PP(V)$ is the space of $1$-dimensional quotients of the vector space $V$.
All the following is standard notation, but needs to be specified in detail
to avoid confusion when constructing the resultants (cf. Lemma \ref{lemmaresult} below).

In the case of $1$-dimensional quotients,  $\PP(V)$ is defined as 
\[ \PP(V) :=  \mathrm{Proj}_{\CC}(\bigoplus_{m\geqslant 0} \Sym^m V) ~,
\]
so   $\oo_{\PP(V)}(1)$ is the tautological quotient line bundle on $\PP(V)$,
and we have Euler exact sequence which corresponds to evaluation of global sections:
\[
0 \to M \to V \otimes \oo_{\PP(V)} \to \oo_{\PP(V)} (1) \to 0 ~.
\]
Furthermore 
 $H^0(\PP(V), \oo _{\PP(V)}(m)) = \Sym^m V$, with
 $\oo_{\PP(V)}(m) = \oo_{\PP(V)}(1)^{\otimes m} $.
 
 Within this notation, to any element 
\[
\sigma \in \Sym^m V = H^0(\PP(V), \oo _{\PP(V)}(m)) ~,
\]
 which is a polynomial of degree $m$, corresponds a zero locus $Z(\sigma) \subseteq \PP(V)$ which is the hypersurface of degree $m$ where the section $\sigma$ vanishes.
 
Fixing a basis $(e_1, \dots, e_r)$ for the vector space $V$,
it induces a dual basis $(\epsilon_1, \dots, \epsilon_r)$, 
then a  point $p \in \pp(V)$ is a one dimensional quotient and therefore a non-vanishing
linear combination $\lambda_1 \epsilon_1 + \dots + \lambda_r \epsilon_r$ (up to multiplication by a scalar).
Elements in $\Sym^m V \cong \Sym^m V^{**}$ are then polynomials to be evaluated on 
elements of $V^*$, 
\emph{i.e.} these are polynomials of degree $m$ in the variables 
$(\lambda_1, \dots , \lambda_r)$. 
%
%This needs to be specified to avoid confusion when constructing the resultants (cf. Lemma \ref{lemmaresult} and the description below for the projectivization of vector bundles).

%we use Grothendieck's  (quotient) notation for projective spaces:
%given a vector space $V$, the projective space $\PP(V)$ is the space of rank one quotients of $V$.
%%, or hyperplanes in $V$.

Given a holomorphic vector bundle $E$,
we denote
$\pi \colon \PP(E) \to X$  the projective bundle of 1-dimensional quotients
of $E$. 
As above  
\[
\PP(E) := \mathrm{Proj}_{\oo_X}(\bigoplus_{m\geqslant 0} \Sym^m E) ~,
\]
and it comes with a Euler exact sequence of vector bundles
\[
0 \to M \to \pi^* E \to \oo_{\PP(E)} (1) \to 0 ~,
\]
where
$\oo_{\PP(E)} (1)$ is a 
line bundle on $\PP(E)$, and 
$\pi^* E \twoheadrightarrow \oo_{\PP(E)} (1)$
is the  \emph{tautological quotient}.

Over a point $x \in X$ we have 
$
\pi^{-1} (x) = \PP(E(x))$, and  
$\oo_{\PP(E)} (1)_{|\pi^{-1} (x)} = \oo_{\PP(E(x))} (1)$
 is the usual 
tautological line bundle quotient of the projective space described above.

%Furthermore,
%any regular map $\varphi \colon Y \to \PP(E)$  is induced by a 
%quotient $f^*E \to L$ with $L = \varphi^* \oo_{\PP(E) (1)}$ a line bundle on $Y$,
%and $f = \pi \circ \varphi$.

Furthermore, there is a canonical isomorphism 
$\pi_* \oo_{\PP(E)} (m) \cong Sym^m E$ which gives on global sections:
\[
H^0(\PP(E),  \oo_{\PP(E)} (m)) \cong H^0(X, Sym^m E) ~.
\]

As above, to any section $\sigma \in  H^0(X, Sym^m E)$ corresponds a section 
$\sigma \in H^0(\PP(E),  \oo_{\PP(E)} (m))$, whose 
zero locus $Z(\sigma) \subseteq \PP(E)$ is a hypersurface, this hypersurface, 
restricted to a fiber $\PP(E(x))$ of $\pi$, is a degree  $m$ hypersurface.

Fixing a local frame   for $E$ over an open subset 
$U \subseteq X$, 
\emph{i.e.} fixing $r$ sections $e_1, \dots, e_r$ of $E$ over  $U$ that provide a basis
$\{e_1(x), \dots, e_r(x) \}$ 
of $E(x)$ 
for each $x \in U$,
we get the 
dual basis $\{ \epsilon_1(x), \dots, \epsilon_r(x) \}$ of $E(x)^*$.
Any element of $\PP(E(x))$ is a  rank one quotient of $E(x)$, so it is
a non-vanishing
linear combination $\lambda_1 \epsilon_1(x) + \dots + \lambda_r \epsilon_r(x)$ up to multiplication by a scalar. 
Any section 
\[
\sigma \in H^0(\PP(E),  \oo_{\PP(E)} (m))\cong H^0(X, Sym^m E) ~,
\]
when restricted to a $U$ can be written as 
\[
\sigma = \sum_{|I| = m} \alpha_I e_1^{i_1}\dots e_r^{i_r}
\]
with $\alpha_I \in \oo_X(U)$.
On each
fiber $\PP(E(x))$, for $x \in U$, the section $\sigma$ corresponds to
 the degree $m$ polynomial 
 \[
 P_x (\lambda_1, \dots , \lambda_r) = \sum_{|I| = m} \alpha_I(x) \lambda_1^{i_1}\dots \lambda_r^{i_r}
 \]
in the variables $(\lambda_1, \dots , \lambda_r)$. 
 
 Changing local frame $f_1, \dots, f_r$ of $E$ over the same open subset $U$, \emph{i.e.} changing trivialization for the vector bundle $E$ over $U$, 
 we have a transition  function $\phi \colon U \to GL_r(\CC)$ such that 
\[
(e_1(x) \dots e_r(x)) = (f_1(x) \dots f_r(x)) \phi(x)
\]
for all $x \in U$.
Then the same section 
\[\sigma \in H^0(\PP(E),  \oo_{\PP(E)} (m))\cong H^0(X, Sym^m E)
\]
will be written locally with respect to the new frame as 
\[
\sigma = \sum_I \beta_I f_1^{i_1}\dots f_r^{i_r}
\]
with $\beta_I \in \oo_X(U)$, corresponding to the polynomial
\[
Q_x (\mu_1, \dots , \mu_r) = \sum_I \beta_I(x) \mu_1^{i_1}\dots \mu_r^{i_r}
\]
on each $\PP(E(x))$.

Now, as 
\[
\sigma = \sum_I \beta_I f_1^{i_1}\dots f_r^{i_r} = 
 \sum_{I} \alpha_I e_1^{i_1}\dots e_r^{i_r} ~,
\]
writing $f^I = f_1^{i_1}\dots f_r^{i_r}$ and 
$f.\phi = (f_1 \dots f_r) \phi$, 
we have:
\[
\sigma = \sum_{I} \alpha_I e^{I} = \sum_{I} \alpha_I (f.\phi)^{I} 
= \sum_{I} \beta_I f^{I} ~.
\]

Then, according to the description above, the two polynomials satisfy:
\[
Q_x (\mu_1, \dots , \mu_r) = P_x ({}^t \phi (x)\cdot (\mu_1, \dots , \mu_r)) ~ \forall x \in U ~,
\]
where the coordinates $(\mu_1, \dots , \mu_r)$ are meant in a column, multiplied by the matrix ${}^t \phi (x)$ on the left.

Letting $x$ vary in $U$ we have an equality
\[
Q_x (\mu_1, \dots , \mu_r) = P_x ({}^t\phi \cdot (\mu_1, \dots , \mu_r)) 
\in \oo_X(U)[\mu_1, \dots , \mu_r]
\]
with 
$\phi \in \mathrm{GL}_{r}(\oo_X(U))$, and ${}^t\phi$ the transpose of the matrix.

\subsection{Previous results}

We recall the following definitions from \cite{mistrurbi} and \cite{mistrettaav}:

\begin{definition}

Let $X$ be a compact complex manifold, and let $E$ be a holomorphic vector bundle.

\begin{enumerate}

\item We say that $E$ is \emph{weakly semiample} if $\oo_{\PP(E)} (m)$ is a base point 
free line bundle on $\PP(E)$ for some $m>0$.

\item We say that $E$ is \emph{strongly semiample} if $Sym^m E$ is  generated 
by its global (holomorphic) sections 
for some $m>0$.

\item We say that $E$ is \emph{asymptotically generically generated}
if there exists a Zariski open subset $U\subset X$ such that $Sym^m E$ is  generated 
by its global sections over the points $x\in U$,
for some $m>0$.

\end{enumerate}

\end{definition}

We want to show that compact parallelizable manifolds and their quotients are related to strong and weak semiampleness of their cotangent bundle.

\begin{remark}

Clearly if $E$ is strongly semiample then $E$ is weakly semiample,
in fact if we have a surjective map
\(
V\otimes \oo_X \twoheadrightarrow Sym^m E
\)
then we have 
\[
V\otimes \oo_{\PP(E)} \twoheadrightarrow \pi^* Sym^m E \twoheadrightarrow \oo_{\PP(E)} (m) ~,
\]
so $\oo_{\PP(E)} (m) $ is globally generated if $Sym^m E$ is globally generated.

Surprisingly the vicecersa does hold only in the case $m=1$:
suppose  the evaluation map 
$\varphi \colon H^0(X,E) \otimes \oo_X \to E$ is \emph{not} surjective on $x\in X$.
Then its image is contained in a hyperplane $H_x$ of $E(x)$:
\[
\varphi (x) \colon H^0(X,E) \to H_x \subset E(x) ~.
\]
This means that on the point 
\[\xi \in \pi^{-1}(x) = \PP(E(x)) \subset \PP(E)
\]
corresponding to the hyperplane $H_x \subset E(x)$, the evaluation map
\[
\psi \colon H^0(\PP(E),  \oo_{\PP(E)} (1)) = 
H^0(X, E) \to \pi^* E_{| \xi} \to  \oo_{\PP(E)} (1)_{|\xi}
\]
vanishes on $\xi$,
therefore $\xi$ is a base point for $\oo_{\PP(E)} (1)$. So 
$\oo_{\PP(E)} (1)$ is globally generated on $\PP(E)$ if and only if 
$E$ is globally generated on $X$.

However the viceversa does not hold for all $m$: we will see below that the cotangent bundle of a hyperelliptic (also called bielliptic) surface is a weakly semiample vector bundle which is not strongly semiample. This is a general property of quotients of compact complex parallelizable manifolds.

\end{remark}

\begin{definition}
Let $X$ be a complex manifold.
We say that $X$ is a 
\emph{parallelizable complex manifold} 
if the holomorphic cotangent bundle $\Omega^1_X$ is trivial, \emph{i.e.} 
if $\Omega^1_X \cong \oo_X^{\oplus n}$.
\end{definition}

\begin{remark}

It is known (cf. \cite{wang}) that a compact complex manifold $X$ is parallelizable if and only if it is isomorphic to a 
quotient $X=H/\Gamma$ of a complex Lie group $H$ by a (cocompact) discrete subgroup 
$\Gamma$.
If the group $H$ is abelian (therefore $X = H/\Gamma$ is a compact complex abelian Lie group) then $X$ is a complex torus.
Also, Wang shows that a compact complex parallelizable
manifold is Kähler if and only if it is a complex torus.

\end{remark}

\begin{remark}

A compact complex manifold whose tangent (or cotangent) bundle is trivial as $\mathcal{C}^{\infty}$ vector bundle is a called \emph{parallelizable differentiable manifold} 
but it is not necessarily a parallelizable complex manifold: for example a Hopf Surface $X$ has $\mathcal{C}^{\infty}$ trivializable tangent space as it is diffeomorphic to 
$S^1 \times S^3$, but it is not a complex parallelizable manifold as 
$h^0(X, \Omega^1_X) = 1$.

\end{remark}

\begin{definition}
We say that a compact Kähler manifold $X$ is \emph{hyperelliptic} if it is a quotient of a complex torus by the free action of a finite group.  
We say that a non-Kähler compact complex manifold is \emph{twisted hyperelliptic} if it is a quotient of a (non-Kähler) compact complex parallelizable manifold (cf. \cite{catanesecorvaja}). 
\end{definition}

\section{Semiample vector bundles on compact complex manifolds}
\label{lemmas}

We prove here Theorem \ref{mainsemiample}, a key ingredient in the characterization of quotients of compact complex parallelizable manifolds.
We state some previous results, which are of interest themselves. 

\begin{proposition}
Let $f \colon X \to Y$ be a surjective map between complex manifolds, and let $E$ be a holomorphic vector bundle on $Y$. Then $E$ is weakly semiample if and only if $f^* E$ is weakly semiample.
\end{proposition}

\begin{proof}
We have the following  fibered product:
\[
\squaremap{\pp(f^* E)}{\widetilde{f}}{\pp(E)}{}{}{X}{f}{Y}
\]
such that $\oo_{\pp(f^* E)}(1) = \widetilde{f}^* \oo_{\pp(E)}(1)$.
Then by 
\cite{fujita}
we have that 
$\oo_{\pp(f^* E)}(1)$ is semiample if and only if 
$\oo_{\pp(E)}(1)$ is semiample.

\end{proof}

\begin{remark}

The statement does not hold if we replace weak semiampleness by strong semiampleness:
even  for finite surjective étale maps $f\colon X \to Y$ we can have  $f^* E$ strongly semiample and $E$ just weakly semiample. We show below that the cotangent bundle of a hyperelliptic surface and its pullback to the abelian surface provide a counterexample.

\end{remark}

The following theorem is proven in \cite{mistrettaparma} 
(cf. Theorem 3.2 and Remark 5.3, \emph{ibid.}):

\begin{theorem}
Let $E$ be a strongly semiample vector bundle of rank $r$ over a compact complex manifold $X$, 
whose determinant has Iitaka-Kodaira dimension  $kod(X, \det E) = 0$.
Then $E\cong L^{\oplus r}$, where $L$ is a torsion line bundle on $X$.
\end{theorem}

As an immediate consequence we have the following

\begin{corollary}
\label{strong}

Let $E$ be a strongly semiample vector bundle of rank $r$ over a compact complex manifold $X$, 
whose determinant has Iitaka-Kodaira dimension  $kod(X, \det E) = 0$.
Then there exists a finite \emph{cyclic} étale covering $\rho \colon \widetilde{X} \to X$ such that 
${\rho}^* E$ is a trival vector bundle. %\cong \oo_{\widetilde{X}}^{\oplus r}$.

\end{corollary}

Another immediate consequence of strong semiampleness of vector bundles is the fact of having a semiample determinant bundle,
in fact if $\Sym^mE$ is globally generated then $\det (\Sym^m E) = (\det E)^{\otimes N}$ is globally generated.
In the weakly semiample case this needs to be proven more carefully. 

Recall that, given a commutative ring $k$, and $r$ homogeneous polynomials
in $r$ variables
$f_1, \dots, f_r \in k[Y_1, \dots, Y_r]$, then the resultant 
$\mathrm{Res}(f_1, \dots, f_r)$ is a well defined element of $k$
 and vanishes if and only if
the polynomials $f_1, \dots, f_r$ have a common non-trivial zero.
    We need some lemmas:

\begin{lemma}[\cite{jouanolou} 5.13]
\label{lemmajouan}
    Let $k$ be a commutative ring, $d_1, \dots , d_r \geqslant 1$ positive integers,
    $f_1, \dots, f_r \in k[Y_1, \dots, Y_r]$ homogeneous polynomials of degrees 
    $d_1, \dots, d_r$. For every matrix $\phi= (a_{ij}) \in \mathcal{M}_{r \times r}(k)$ 
    and every polynomial $f \in k[Y_1, \dots, Y_r]$ set
    \[
    f \circ \phi = f(a_{11}Y_1+a_{12}Y_2 + \dots + a_{1n}Y_N, \dots,
    a_{n1}Y_1+a_{n2}Y_2+ \dots + a_{nn}Y_N)=
    \]
    \[
=    f(\phi (Y_1, \dots , Y_r)) ~.
    \]
Then the resultant satisfies the following equality:
\[
\mathrm{Res}(f_1\circ \phi, \dots, f_r\circ \phi)=
\det (\phi)^{d_1d_2\dots d_r} \mathrm{Res}(f_1, \dots, f_r) ~.
\]

\end{lemma}

\begin{lemma}
\label{lemmaresult}

Let $E$ be a rank $r$ holomorphic  vector bundle on a compact complex manifold $X$. 
Consider r sections $\sigma_1, \dots, \sigma_r$, with 
$\sigma_j \in H^0(\pp(E), \oo_{\pp(E)}(m_j))$, and $m_j>0$ positive integers.

Then there is a section 
$\mathrm{Res}(\sigma_1, \dots, \sigma_r) \in H^0(X, \det(E)^{\otimes m_1\cdot m_2 \dots m_r})$ 
vanishing exactly on the points $x\in X$ such that the sections
$ \sigma_{1 |\pp(E(x))} , \dots , \sigma_{r |\pp(E(x))}$ have a common zero in 
$\pp(E(x))$.
\end{lemma}

\begin{proof}
   Let us define the  section $\mathrm{Res}(\sigma_1, \dots, \sigma_r)$ of
   $H^0(X, \det(E)^{\otimes m_1\cdot m_2 \dots m_r})$ locally using charts.
   For this purpose, we fix an open covering $\{ U_\alpha \}$ of the manifold $X$ for which there are local trivializations 
   \[
   \left\{
   \begin{tikzcd}
       \varphi_\alpha : E_{|U_\alpha} \ar[r, "\cong"] & \CC^r \times U_\alpha
   \end{tikzcd}
   \right\}
   \]
   of the vector bundle $E$ over the open subsets $U_\alpha$. Moreover, let
   \[
   \left\{
   \begin{tikzcd}
       \phi_{\alpha\beta} : U_{\alpha\beta} \ar[r] & \GL_r(\CC)
   \end{tikzcd}
   \right\}
   \]
   be the set of transition functions attached to the local trivializations $\varphi_\alpha$, where $U_{\alpha\beta} := U_\alpha \cap U_\beta$; i.e., one has
   $(\varphi_\beta)_{|U_{\alpha\beta}}
   \circ 
   (\varphi_\alpha)_{|U_{\alpha\beta}}^{-1} (v, p) = (\phi_{\alpha\beta} (p) \cdot v , p )$, for $ (v, p) \in \CC^r \times U_{\alpha \beta}$.
   
   Let $e^{\alpha}_1, \dots, e^{\alpha}_r$ be a local frame corresponding to $\varphi_{\alpha}$, and let  $Y_{1, \alpha}, \ldots, Y_{r, \alpha}$ be variables corresponding to dual coordinates in the chart $\varphi_\alpha$. The global sections $\sigma_1, \ldots , \sigma_r$ correspond in the local chart $\varphi_\alpha$ to homogeneous polynomials $f_{1, \alpha}(Y_{1, \alpha}, \ldots, Y_{r, \alpha})$, $\ldots$ , $f_{r, \alpha}(Y_{1, \alpha}, \ldots, Y_{r, \alpha})$ of degrees $m_1, \dots m_r$
   with coefficients in $\oo_X(U_{\alpha})$, as seen in Section \ref{projectivization} above.
   
Furthermore, for $\varphi_\beta$ another local chart, and for corresponding local frame
$e^{\beta}_1, \dots, e^{\beta}_r$, one has 
\[
(e_1^{\alpha}(x) \dots e^{\alpha}_n(x)) = (e^{\beta}_1(x) \dots e^{\beta}_r(x)) \phi_{\alpha\beta}(x)
\]
 for all $x \in U_{\alpha \beta}$, and the
 global sections $\sigma_1, \ldots , \sigma_r$ correspond in the local chart $\varphi_\beta$ 
 to homogeneous  polynomials $f_{1, \beta}(Y_{1, \beta}, \ldots, Y_{r, \beta})$, $\ldots$ , $f_{r, \beta}(Y_{1, \beta}, \ldots, Y_{r, \beta})$
  with coefficients in $\oo_X(U_{\beta})$. 
  Then, according to computation of Section \ref{projectivization}
 we have:
 \[
f_{i, \beta} (Y_{1, \beta}, \ldots, Y_{r, \beta}) = 
f_{i,\alpha}({}^t\phi_{\alpha \beta} (Y_{1, \beta}, \ldots, Y_{r, \beta}))
 \]
   over $U_{\alpha \beta}$.

Now on each open subset $U_{\alpha}$ we have functions
\[
R_{\alpha}=\mathrm{Res}(f_{1, \alpha}(Y_{1, \alpha}, \ldots, Y_{r, \alpha}), \ldots , f_{r, \alpha}(Y_{1, \alpha}, \ldots, Y_{r, \alpha})) \in \oo_X(U_{\alpha}),
\]
and by Lemma \ref{lemmajouan}
we have for each $\alpha, \beta$:
\[
R_{\beta} = \det (\phi_{\alpha \beta})^{m_1 \dots m_r} R_{\alpha}
\]
on $U_{\alpha \beta}$. Then the functions $R_{\alpha} \in \oo_X(U_{\alpha})$
 glue to a global section 
 \[
 \mathrm{Res}(\sigma_1, \dots, \sigma_r) \in H^0(X,\det (E)^{m_1\dots m_r})
 \]
 which has the required properties.

\end{proof}

   % Furthermore, for $\varphi_\beta$ another local chart, one has at a point $z\in U_{\alpha\beta}$
   % \[
   % (Y_{1, \alpha}, \ldots, Y_{r, \alpha}) = (Y_{1, \beta}, \ldots, Y_{r, \beta})\phi_{\alpha\beta}(z) \ .
   % \]
   % Writing the above equation as 
   % $\overline{Y}_\alpha = \Phi_{\alpha\beta}\overline{Y}_\beta$, one has
   % \[
   % f_{i, \beta}(\overline{Y}_\beta) = 
   % f_{i, \alpha}(\Phi_{\alpha\beta}\overline{Y}_\beta) \ \mathrm{for} \ i = 1, \ldots ,  k \ .
   % \]
   % Let $\mathrm{Res}(f_{1, \alpha}, \ldots , f_{r, \alpha})$ be the resultant of the polynomials $f_{1, \alpha}, \ldots , f_{r, \alpha}$ (see \cite[Section 2]{jouanolou}). 
   % Fixing a point $x\in U_\alpha$, it is a complex number which vanishes if and only if the polynomials $f_{1, \alpha}, \ldots , f_{r, \alpha}$ admit a common notrivial zero, that is if and only if the sections $\sigma_1, \ldots , \sigma_r$ admit a common zero in $\pp(E(x))$.
   % Then, by the above equality and \cite[(5.13.1)]{jouanolou}, one concludes that the local sections $\{\mathrm{Res}(f_{1, \alpha}, \ldots , f_{r, \alpha})\}$ patch together to give a global section $\mathrm{Res}(\sigma_1, \ldots , \sigma_r)$ of the line bundle $\det(E)^{\otimes m_1\cdot m_2 \dots m_r}$ having the stated property.

\begin{theorem}

Let $E$ be a weakly semiample vector bundle  over a compact complex manifold $X$, 
then $\det E$ is a semiample line bundle. 
More precisely, if $\oo_{\pp(E)}(m)$ is base point free on $\PP(E)$, then 
$\det (E)^{\otimes m^r}$ is base point free on $X$.

\end{theorem}

\begin{proof}

As $E$ is weakly semiample, let us consider the holomorphic map
\[
\varphi \colon \PP(E) \to \PP^N = \PP (H^0(\pp(E), \oo_{\pp(E)}(m))) 
\]
induced by the complete linear system
$H^0(\pp(E), \oo_{\pp(E)}(m))$ which is base point free  for a suitable $m>0$.

Let us fix $x \in X$, as $\oo_{\PP(E)} (m)_{|\pi^{-1} (x)} = \oo_{\PP(E(x))} (m)$, and 
$\varphi_{|\PP(E(x))}$ is the map associated to the base point free linear system 
\[|W_x| = | \mathrm{Im} (H^0(\pp(E), \oo_{\pp(E)}(m)) \to H^0(\pp(E(x)), \oo_{\PP(E(x))}(m)))|\]
on $\PP(E(x))$, then $\varphi_{|\PP(E(x))}$ is finite, therefore 
$\dim  \varphi ({\PP(E(x))}) = \dim \PP(E(x)) = r-1$, where $r = \rk E$.
So we can find $r$
hyperplanes in $\PP^N$ whose intersection does not meet $\varphi ({\PP(E(x))})$.
This means that $r$ general sections 
$\sigma_1, \dots, \sigma_r \in H^0(\pp(E), \oo_{\pp(E)}(m))$ satisfy
$Z \cap \pi^{-1} (x) = \emptyset$, where 
$Z = Z(\sigma_1, \dots, \sigma_r)$ is the locus of their common zeros in $\PP(E)$.

Now let $V = \pi (Z)$ be the image of this locus  in $X$,
then  Lemma \ref{lemmaresult} states that 
$V = Z(\mathrm{Res}(\sigma_1, \dots, \sigma_r)) \subseteq X$ is a divisor in 
$|\det (E)^{\otimes m^r}|$, and $x \notin V$.
Therefore for $x \in X$ we  find a divisor in 
$|\det (E)^{\otimes m^r}|$ whose support does not contain $x$.
As for each $x \in X$ we can find such a divisor 
in $|\det (E)^{\otimes m^r}|$ not containing $x$,
then
$\det (E)^{\otimes m^r}$ is base point free.

\end{proof}

\begin{remark}

The locus $V \subseteq X$  in the proof above can be empty,
this happens if $\dim H^0(\pp(E), \oo_{\pp(E)}(m)) = \rk E$, 
for example when $E$ is trivial, 
or in the case of the theorem below. In this case 
$\det (E)^{\otimes m^r}$ is trivial.

\end{remark}

\begin{remark}

This proof follows Fujiwara's ideas contained in  \cite{fujiwara},
adapting his constructions to the compact complex case, where we cannot make use of Chern
 classes computations in the Chow ring, through the use of the resultant.
The following theorem as well follows Fujiwara's constructions.

\end{remark}

As a consequence of the semiampleness of the determinant, we can show the following

\begin{theorem}
\label{semiample}
Let $E$ be a weakly semiample vector bundle of rank $r$ over a compact complex manifold $X$, 
whose determinant has Iitaka-Kodaira dimension  $kod(X, \det E) = 0$.
Then there exists a finite Galois étale covering $\rho \colon \widetilde{X} \to X$ such that 
$\rho^* E$ is a trivial vector bundle.

\end{theorem}

\begin{proof}

Let us remark that within these hypotheses the line bundle 
$\det E$ is a torsion line bundle, in fact it is semiample and of Iitaka dimension $0$.
Now consider the map 
\[
\Phi \colon \PP(E) \to \PP^N = \PP (H^0(\pp(E), \oo_{\pp(E)}(m))) ~,
\]
then we can show that $\dim \Phi (\PP(E)) = r-1$:
if by contradiction the image were of dimension at least $r$, then it would have 
a non-empty intersection with
$r$ hyperplanes in $\PP^N$ corresponding to 
$r$ sections 
$\sigma_1, \dots, \sigma_r \in H^0(\pp(E), \oo_{\pp(E)}(m))$.
Then the zero locus
$Z(\mathrm{Res} (\sigma_1, \dots, \sigma_r)) \subset X$ 
would be a nonempty divisor in $|\det (E)^{\otimes m^r}|$,
contradicting the fact that $\det (E)$ is a torsion line bundle.
We remark that the restriction of $\Phi$ to each fibre 
$\pi^{-1}(x) = \PP(E(x))$ is given by the base point free linear system 
\[
W_x = \mathrm{Im} (H^0(\pp(E), \oo_{\pp(E)}(m)) \to H^0(\pi^{-1}(x), \oo_{\pp(E(x))}(m)))
\]
and is therefore a finite map 
$\Phi_{|\PP(E(x))} \colon \PP(E(x)) \to \Phi (\PP(E))$, which is a projection of a Veronese embedding. In particular it follows that $\dim \Phi (\PP(E))$ is exactly $r-1$.

Then a general fiber of the map $\Phi$ will be 
a disjoint union of 
compact complex submanifolds of $\PP(E)$ 
of dimension $n = \dim X$, dominating $X$ with finite maps.
In fact if the map 
\[
\pi_{|\Phi^{-1}(y)} \colon 
\Phi^{-1} (y) \to X
\]
were not surjective and finite, we would have a curve $C\subseteq \PP(E)$
contracted  by $\pi$ and by $\Phi$, 
and this is not possible as $\Phi_{|\PP(E(x))} $ is finite.

Let us call $Z$ one of these components dominating $X$,
and let us show that the map $Z \to X$ is a finite étale
covering, and that there is further finite étale 
covering $\widetilde{Z} \to Z$ such that 
$g \colon \widetilde{Z} \to X$ satisfies
\[
g^* E = \oo_{\widetilde{Z}} \oplus E^{\prime} ~.
\]

First, notice that as $Z$ is contracted by $\Phi$ then 
\[
\oo_{\PP(E)} (m)_{|Z} = \oo_Z ~,
\]
and that 
\(
\omega_Z = \omega_{\PP(E)} \otimes \oo_Z , \textrm{where }
\omega_{\PP(E)} = \oo_{\PP(E)}(-r-2) \otimes \pi^* (\omega_X \otimes \det E) ~.
\)

As $(\det E )^{\otimes K} = \oo_X$ for some $K>0$, we obtain that 
\[
\omega_Z^{\otimes Km} \cong \pi^* \omega_X^{\otimes Km} ~,
\]
but as $\pi$ is a finite map there is an injective morphism
$\pi^* \omega_X \to \omega_X$,
therefore we get $\omega_Z \cong \pi^*\omega_X$ and so 
$\pi \colon Z \to X$ is an unramified finite covering where 
$\oo_{\PP(E)} (m)_{|Z} \cong \oo_Z$.
 Using this last equality we see that we can find another étale covering 
$\widetilde{Z} \to Z$,
such that 
$\oo_{\PP(E)} (1)_{|\widetilde{Z}} \cong \oo_{\widetilde{Z}}$.

 Furthermore, the universal quotient $\pi^* E \twoheadrightarrow \oo_{\PP(E)}(1)$ pulls back to 
 a quotient $h \colon g^* E \twoheadrightarrow \oo_{\widetilde{Z}}$ on $\widetilde{Z}$, 
 and this quotient $h$ is split.
In fact from  the construction of $\widetilde{Z}$ it follows that the surjective map 
 $\Sym^m h \colon g^*\Sym^m E \to \oo_{\widetilde{Z}}$ splits, and this implies that $h$ splits as well: we can factor the map $h$ as
\[
 g^*\Sym^m E \to g^*\Sym^{m-1} E \to \oo_{\widetilde{Z}} ~,
\]
the first map being
\[
v_1\cdot v_2 \dots v_m \mapsto \sum \frac{h(v_j)}{m} v_1\cdot v_2 \dots \hat{v_j} \dots v_m ~,
\]
and the second  $\Sym^{m-1} h$, then the splitting of $\Sym^m h$ implies that $\Sym^{m -1}h$ splits as well and we proceed by recursive induction on $m$. 

Now we constructed a finite étale 
covering 
$g \colon \widetilde{Z} \to X$ such that
\(
g^* E = \oo_{\widetilde{Z}} \oplus E^{\prime} 
\), 
so $ E^{\prime}$ is again weakly semiample, 
and by recursive induction on $r = \rk E$ we obtain a finite étale map 
$\dbtilde{Z} \to X$ such that $E$ pulls back to a trivial vector bundle. 
 
Finally, by considering the Galois closure of the étale finite covering 
$\dbtilde{Z} \to X$ we obtain a {finite} étale  morphism
$\widetilde{X} \to \dbtilde{Z} $ such that 
$\rho \colon \widetilde{X} \to \dbtilde{Z} \to X$ 
is a finite Galois étale covering and that $\rho^* E$ is a trivial bundle (cf. Remark \ref{galois} below).

\end{proof}

\begin{remark}
\label{galois}

Given a finite étale covering space $Z \to X$, we can always find a further finite étale covering 
$Z^{\prime} \to Z$ such that $Z^{\prime} \to X$ is Galois. In fact if the covering $Z \to X$ corresponds 
to the finite index subgroup $H \leq \pi_1(X)$, then we can consider the covering $Z^{\prime} \to X$
corresponding to the maximal normal subgroup $H^{\prime} \trianglelefteq \pi_1(X)$ contained in $H$, 
and it can be shown that this has finite index in $\pi_1(X)$. The map $Z^{\prime} \to Z \to X$ is called 
Galois closure of the covering $Z \to X$.
\end{remark}

\begin{remark}

With the hypothesis of Theorem \ref{semiample} above, the determinant line bundle 
$\det E$ is semiample and of Kodaira dimension $0$, so it is a torsion line bundle.
Therefore there exists a finite cyclic Galois étale cover that trivializes the line bundle 
$\det (E)$,
however in general this is not the cover that trivializes $E$ itself, we might need to consider more complicated covers as it appears in the proof of 
the theorem.
In fact the cover trivializing $E$ needs not to be cyclic if $E$ is weakly semiample, while it is cyclic in the strongly semiample case.
    
\end{remark}

\section{Parallelizable Manifolds}
\label{main1}

In this section we give a characterization of parallelizable manifolds and their quotients.

An immediate consequence of Theorem \ref{semiample} above is the following

\begin{theorem}
\label{hyperell2}

Let $X$ be a compact complex manifold. Then $X$ is a (possibly twisted) hyperelliptic manifold if and only if 
it has Kodaira dimension $k(X) = 0$ and weakly semiample cotangent bundle.

\end{theorem}

\begin{proof}
If $X$ is hyperelliptic then it is the quotient of a parallelizable manifold under the free action of a finite group, then $X$ has Kodaira dimension $0$ and $\Omega_X^1$ pulls back to a trivial bundle, 
therefore $\Omega^1_X$ is weakly semiample.

Viceversa, applying Theorem \ref{semiample} to the holomorphic cotangent bundle $\Omega^1_X$
we know that there exists a  finite Galois étale cover 
$f \colon \widetilde{X} \to X$ such that 
$f^* \Omega^1_X$ is a trivial vector bundle.
As $f$ is étale then $f^* \Omega^1_X \cong \Omega^1_{\widetilde{X}}$,
so $\widetilde{X}$ is a compact complex  parallelizable manifold and $X$ is a (possibly twisted) hyperelliptic manifold.
\end{proof}

The following corollaries are immediate:

\begin{corollary}

A compact K\"ahler manifold $X$ is hyperelliptic if and only if it has Kodaira dimension $0$ and weakly semiample cotangent bundle.

\end{corollary}

\begin{corollary}

A smooth projective variety $X$ over the complex numbers is isomorphic to the quotient of an abelian variety by the free action of a finite group if and only if it has Kodaira dimension $0$ and weakly
 semiample cotangent bundle.

\end{corollary}

It can be shown that a  proper algebraic variety over the complex numbers
is a  compact complex parallelizable manifold if and only if it is an abelian veriety
(cf. \cite{winkelmann}),
so we have the following:

\begin{corollary}

A smooth proper algebraic variety $X$  over the complex numbers is isomorphic to the quotient of an abelian variety by the free action of a finite group if and only if it has Kodaira dimension $0$ and weakly
 semiample cotangent bundle.

\end{corollary}

Next we want to prove that in a similar way we can characterize compact complex parallelizable manifolds as those  having Kodaira dimension $0$ and \emph{strongly} semiample cotangent bundle.

First we need a lemma on the structure of the group of 
automorphisms of a parallelizable manifold: 
we remark that a compact complex parallelizable manifold is a quotient 
$P = H / \Gamma$
of 
a complex Lie group $H$ by a (cocompact) discrete 
subgroup $\Gamma$ (cf. \cite{wang}). 
It can be 
shown that any holomorphic map of parallelizable manifolds, 
up to translations, comes
from a homomorphism of complex Lie groups:

\begin{lemma}
\label{autom}

Consider two compact complex parallelizable manifolds 
$P = H / \Gamma$ and $Q = H^{\prime} / \Gamma^{\prime}$,
and let 
$f \colon P \to Q$ be a holomorphic map.
Then there exists a homomorphism of complex Lie groups
$F \colon H \to H^{\prime}$, with $F(\Gamma) \subseteq \Gamma^{\prime}$,
and an element 
$q \in H^{\prime}$, such that for all 
$h \in H$:
\[
f(h \Gamma) = q  F(h) \Gamma^{\prime} ~.
\]

\end{lemma}

\begin{proof}

Cf. \cite{catanesehyperII}
and \cite{winkelmann}.
\end{proof}

Next we can show that a group freely acting on a parallelizable manifold
 so that the quotient has strongly semiample cotangent bundle must act
 trivially on holomorphic differential $1$-forms:
 
\begin{lemma}
\label{translations}
Let $P = H / \Gamma$ be a compact complex parallelizable manifold,
and let $G$ be a finite group acting on $P$ so that 
$X = P/G$ has strongly semiample cotangent bundle.
Then $G$ acts trivially on $H^0(P, \Omega^1_P)$,
\emph{i.e.} for each $g \in G$ the map 
$g^* \colon H^0(P, \Omega^1_P) \to H^0(P, \Omega^1_P)$ is the identity map.

\end{lemma}

\begin{proof}

Let $g \in G$ act on $P = H/\Gamma$, let us show that 
if $\Omega^1_X$ is strongly semiample then $g^*$ is the identity homomorphism on 
$H^0(P, \Omega^1_P)$.

As the quotient $\rho \colon P \to X \cong P/G$ is étale, then 
$\rho^* \Omega^1_X = \Omega^1_P$, 
and we have a natural isomorphism for each $m>0$:
\[
H^0(X, \Sym^m \Omega^1_X) \cong H^0(P, \Sym^m \Omega^1_P)^G ~.
\]
However $\Omega^1_P$ (so $\Sym^m \Omega^1_P$) is a trivial vector bundle, and 
the action of $G$  on 
\[
H^0(P, \Sym^m \Omega^1_P) = \Sym^m H^0(P,\Omega^1_P)
\]
 is the symmetric power 
of the action of $G$ in $H^0(P, \Omega^1_P)$.

Now, let $K>0$ be an integer such that 
$\Sym^K \Omega^1_X$ is globally generated, then 
\[
\dim H^0(X, \Sym^K \Omega^1_X) = \dim H^0(P, \Sym^K \Omega^1_P)^G
\geqslant \rk \Sym^K \Omega^1_X = 
\]
\[
= \rk \Sym^K \Omega^1_P = \dim H^0(P, \Sym^K \Omega^1_P) ~,
\]
so all elements in $\Sym^K H^0(P,  \Omega^1_P)$ are invariant by
the action of $G$.

As $G$ is a finite group, any $g \in G$ acts on 
$H^0(P, \Omega^1_P)$ in a diagonalizable way, then it is easily seen 
that triviality of the action on $\Sym^K H^0(P,  \Omega^1_P)$
implies that $G$ acts by homotheties on $H^0(P, \Omega^1_P)$
(cf. \cite{mistrurbi}, proof of Theorem 4.14).
So there exists a character 
$\chi \colon G \to \CC^*$ such that each $g\in G$ acts on $H^0(P, \Omega^1_P)$
as $\chi_g Id$.  Let us show that this character is trivial.

Let us call $\mathfrak{h} = T_{e}H$ the Lie algebra of the complex Lie group $H$, then 
we have $\mathfrak{h}  \cong  H^0(P, \mathcal{T}_P)$, where $\mathcal{T}_P$ is the (trivial) tangent bundle on $P$, so
\[
\mathfrak{h}^*  =  H^0(P, \mathcal{T}_P)^* \cong   H^0(P, \Omega^1_P)
\]
and 
\[
g_* = {}^t(g^*) = \chi_g Id_{\mathfrak{h}} \colon \mathfrak{h} \to \mathfrak{h}
\] is a 
Lie algebra homomorphism.

So we must have, for each 
$x, y \in \mathfrak{h}$:
\[
\chi_g [x,y] = g^*([x,y]) = [g^* x, g^*y] = [\chi_g x , \chi_g y] = \chi_g^2 [x,y] ~.
 \] 
This can happen only if either $\chi_g = 1$, or if
$[x,y] = 0$ for all $x,y \in \mathfrak{h}$.
The first case means that 
\(g^* = \chi_g Id \colon H^0(P, \Omega^1_P) \to H^0(P, \Omega^1_P) \) is the identity map, 
while the second case means that the Lie algebra $\mathfrak{h}$ is abelian.
But as $P = H / \Gamma$ if $\mathfrak{h}$ is abelian  then $H$ is an abelian complex Lie group, and therefore $P$ is a complex torus ($P$ is compact).
If $P$ is a complex torus we can  show  that $\chi_g =1$ as well. In fact, 
 writing $g \colon P \to P$ as
\[g (x) = (\chi_g)  x + v ~,\]  
if $\chi_g \neq 1$ then the point $x = (1 - \chi_g)^{-1} v$ is a fixed point
of $g$, which we cannot have as the action is free.

\end{proof}

\begin{remark}[Catanese, personal communication]

A more general argument of Fabrizio Catanese can be applied to show that  $G$ acts trivially on $H^0(P, \Omega^1_P)$ in Lemma \ref{translations} above:
a holomorphic automorphism $g \colon P \to P$ of a compact complex parallelizable manifold $P = H/\Gamma$
such that all eigenvalues of $g^* \colon H^0(P, \Omega^1_P) \to  H^0(P, \Omega^1_P)$ 
are different from $1$  must have a fixed point.
In fact for each $q \in H$ the map $g$ is homotopic to the translated map $qg$. 
Then the intersection number $\nu$ in $P\times P$ of the diagonal and the graph of $g$, 
is the same for the map $g$ and for the map $qg$, and is equal to $0$ if $g$ has no fixed point. As one can find $q\in H$ such that 
$qg$ has a fixed point, and as this is an isolated fixed point because all eigenvalues are different from $1$, then $\nu >0$, so $g$ must have a fixed point as well. 
In the proof of Lemma \ref{translations} above, 
as $g$ has no fixed point we deduce that the map 
$g^* = \chi_g Id \colon H^0(P, \Omega^1_P) \to H^0(P, \Omega^1_P)$
 has all eigenvalues equal to $1$, so $\chi_g =1$.

\end{remark}

We can use these constructions to characterize parallelizable manifolds through strong semiampleness:

\begin{theorem}
\label{mainparall}

Let $X$ be a compact complex manifold. Then $X$ is parallelizable if and only if 
it has Kodaira dimension $k(X) = 0$ and strongly semiample cotangent bundle.

\end{theorem}

\begin{proof}

Clearly a  compact complex parallelizable manifold has Kodaira dimension zero and strongly semiample cotangent bundle.

Conversely, let us suppose that $X$ has Kodaira dimension $0$ and strongly semiample 
cotangent bundle.
Then we know that $X = P /G$ is the quotient of a compact complex parallelizable manifold $P$ by the free action of a finite group $G$.

So  according to Lemma \ref{translations} each element of $G$ acts trivially 
on $H^0(P, \Omega^1_P)$, \emph{i.e.} we have 
\[
H^0(P, \Omega^1_P)^G = H^0(P, \Omega^1_P) ~.
\]

As $X$ has Kodaira dimension $0$, in order to show that $X$ is parallelizable we 
have to show that $\Omega^1_X$ is globally genereated.

Now fix a point $x \in P$ and its image $\pi(x)= \bar{x} \in X$.
As $\pi$ is an étale map, we can consider the following commutative diagram:

\[
\squaremap{H^0(X, \Omega^1_X)}{ev_{\bar{x}}}{\Omega^1_{X, \bar{x}}}{\pi^*}{{}^t (d_x \pi)}
{H^0(P, \Omega^1_P)}{ev_x}{\Omega^1_{P, x}} 
\]

where the vertical map on the left is an isomorphism 
because $H^0(P, \Omega^1_P)^G = H^0(P, \Omega^1_P)$ and 
the vertical map on the left is an isomorphism as 
$\pi \colon P \to X = P/G$ is étale, while 
the map $ev_x \colon H^0(P, \Omega^1_P) \to \Omega^1_{P, x}$ is an isomorphism 
as $P$ is compact parallelizable.
So the map $ev_{\bar{x}} \colon H^0(X, \Omega^1_X) \to \Omega^1_{X, {\bar{x}}}$
is an isomorphism, hence $X$ is parallelizable.

\end{proof}

\begin{remark}

As it is shown in \cite{mistrettaparma}, 
a vector bundle $E$ on a compact complex manifold $X$ is strongly semiample 
and has determinant $\det E$ with Kodaira-Iitaka dimension $0$ if and only if
it is a direct sum of isomorphic torsion line bundles
$E = L \oplus L \oplus \dots \oplus L$. Now when $E$ is the cotangent bundle
it cannot happen that $\Omega^1_X$ is the direct sum of $n$ copies of a non trivial torsion line bundle: in that case $\Omega^1_X$ is strongly semiample 
and has determinant $\omega_X$ with Kodaira-Iitaka dimension $0$, so it is trivial
because of Theorem \ref{mainparall}.

\end{remark}

\begin{remark}
\label{translations2}

Let $P = H / \Gamma$ be a compact complex parallelizable manifold,
then we know that for each biholomorphic map $g\colon P \to P$
there exist an element $q \in H$ and an automorphism $F\colon H \to H$ of complex Lie groups, such that
\[
g(h \Gamma) = (q F(h)) \Gamma ~.
\]
Also, for each automorphism $F\colon H \to H$ of complex Lie groups such that 
$F(\Gamma) \subseteq \Gamma$ and for each $q \in H$  the map defined as above
\[
g(h \Gamma) = (q F(h)) \Gamma 
\]
is a biholomorphic map $P \to P$.

Now let $q\in H$ be an element such that 
$q^{-1} \Gamma q = \Gamma$, and let $F \colon H \to H$ be the map
$F(h) = q^{-1} h q$. Then the translation by $q$ on the right is well defined as
a biholomorphic map $g \colon P \to P$, and can be described as above by
\[
g \colon h\Gamma \mapsto (qF(h))\Gamma = (hq)\Gamma ~.
\]
In this case the map $F \colon H \to H$ is not necessarily the identity map on $H$,
however the map of Lie algebras (the spaces of holomorphic vector field on $P$, which is naturally isomorphic to  the Lie Algebra $\mathfrak{h}$ of the Lie group $H$)
\[
g_*: H^0(P, \mathcal{T}_P) = \mathfrak{h}  \to H^0(P, \mathcal{T}_P) = \mathfrak{h}
\]
is the identity map, therefore the map
\[
g^*  \colon H^0(P, \Omega^1_P) \to  H^0(P, \Omega^1_P)
\]
is the identity map as well.
In fact if we call $\el_q \colon P \to P$ multiplication by $q$ on the left,
we have that $g = \el_q \circ F \colon P \to P$.
Then we have $(\el_q)_* = Ad_q$ and $F_* = Ad_{q^{-1}}$, so $g_*$ and $g^*$ are the identity maps.

Conversely, suppose we have an isomorphism $g \colon P \to P$ such that 
\[
g^* = id \colon H^0(P, \Omega^1_P) \to  H^0(P, \Omega^1_P) ~.
\]
Then if we write $g(h \Gamma) = qF(h) \Gamma$
with $F \colon H \to H$ an isomorphism of complex Lie groups and $q \in H$
 we see that 
$\el_q^* = (F^*)^{-1}$, so 
\[
F^* = (\el_q^*)^{-1} = {^t}(Ad_{q^{-1}})
\]
therefore if we suppose that $H$ is a simply connected complex Lie group 
we have that the Lie group homomorphism $F$ is the conjugation map 
$F(h)= q^{-1} h q$, and $q \in H$ is an element such that 
$q^{-1}  \Gamma q = \Gamma$ because the map $F$ induces a map  $g \colon P \to P$.

In Theorem \ref{mainparall} above, we have $P = H/ \Gamma$ and $X = P/G$ with $G$ a finite cyclic group of order $m$,
generated by $g$. Then, as we know that $g^* = id$ on $H^0(P, \Omega^1_P)$,
the biholomorphic map $g$ must be $\el_q \circ Ad_{q^-1}$. So we have that 
 \[
 g (h\Gamma) = q (q^{-1}hq)\Gamma = (hq) \Gamma ~,
 \]
 and so 
 \[X = P / G = (H/\Gamma)/G = H  / <q, \Gamma> ~.
 \]
 
 Furthermore we have that $q \in H$ is an element such that 
 $ q^{-1} \Gamma q = \Gamma$, such that  
 \[
 q^m \in \bigcap_{h\in H} h\Gamma h^{-1}
 \]
 (where $m$ is the order of the cyclic group $G = <g>$) and such that 
 \[
 q^k \notin  h\Gamma h^{-1} \textrm{ for any } h \in H 
 \textrm{ and any } k=1, \dots , m-1 ~.
 \]

%
%and let $G$ be a finite group acting on $P$ so that 
%$X = P/G$ has strongly semiample cotangent bundle.
%Then $G$ acts  on $P$ by translations on the right, \emph{i.e.} for each
%$g\in G$ there exists
%$q_g \in H$ such that 
%\[
%g(h \Gamma) = (h \cdot q_g) \Gamma ~.
%\]
%
%In fact for each $g \in G$ acting on $P = H/\Gamma$ there exist $q\in H$ and 
%an automorphism of complex Lie groups 
%$F \colon H \to H$ such that
%\[
%g(h \Gamma) = (q F(h)) \Gamma 
%\].
%
%
%
%
%
%
%
%Let $g \in G$ act on $P = H/\Gamma$, then 
%there exists a Lie group  homomorphism $F \colon H \to H$ and 
%an element $q \in G$ such that 
%$g(h \Gamma) = q \cdot F(h) \Gamma$. Now we know that 
%

\end{remark}

%
%\begin{proof}
%
%
%Let us suppose now that $X$ is a compact complex manifold
%with Kodaira dimension $k(X) = 0$ and strongly semiample cotangent bundle.
%Applying Corollary \ref{strong},
%we know that there exists a finite étale (cyclic) covering 
%$\rho \colon \widetilde{X} \to X$
%such that $\rho^* \Omega^1_X$ is trivial.
%As $\rho$ is étale, then $\rho^* \Omega^1_X \cong \Omega^1_{\widetilde{X}}$,
%therefore $\widetilde{X} = P$ is a parallelizable manifold,
%and $X$ is a quotient  of $P$ by the  action of a finite cyclic group $G$ acting freely on $P$.
%
%Let us show that $X$ itself is a parallelizable manifold:
%according to Lemma \ref{translations2}, the  group
%$G$ acts on $P$ by translations on the right, so 
%\[
%X \cong P/G \cong (H/\Gamma)/G
%\]
%and
%$G$ is a finite cyclic group generated by a translation of $P$, \emph{i.e.} by an element
% $q \in H$ of finite order as a translation in $H/\Gamma$. 
%So $X \cong H/ <\Gamma, q>$   is a parallelizable manifold itself.
%
%\end{proof}
%

As any compact complex parallelizable manifold which is K\"ahler is a complex torus,
 as any compact complex parallelizable manifold which is projective is an abelian variety, and as any smooth proper algebraic variety over the complex numbers which is parallelizable is an abelian variety   we have the following corollaries:

\begin{corollary}

A compact K\"ahler manifold $X$ is biholomorphic to a complex torus if and only if it has Kodaira dimension $0$ and strongly semiample cotangent bundle.

\end{corollary}

\begin{corollary}

A smooth projective variety $X$ over the complex numbers is isomorphic to an abelian variety if and only if it has Kodaira dimension $0$ and strongly semiample cotangent bundle.

\end{corollary}

\begin{corollary}

A smooth proper algebraic variety $X$  over the complex numbers  is isomorphic to an abelian variety if and only if it has Kodaira dimension $0$ and strongly semiample cotangent bundle.

\end{corollary}

\section{Final remarks and questions}
\label{conclusion}

Since the times of K\"ahler (cf. \cite{kahler}) and Severi 
(cf. \cite{severi} and \cite{severi2}), 
 a lot of interesting geometrical and topological 
properties were discovered relating differential forms and geometry.
We address first some questions concerning the fundamental group of varieties with 
semiample cotangent bundle.

\subsection{Manifolds with infinite fundamental group}

We mention in particular some results and conjectures 
on the relationship between the 
existence of holomorphic symmetric differentials and the topology 
of a complex manifold:

\begin{conjecture}[Mumford]
Let $X$ be a compact K\"ahler manifold. Then $X$ is rationally connected 
if and only if $H^0(X, \Sym^m \Omega_X^p) =0$ for all $m>0$ and all $p>0$.

\end{conjecture}

Related to this conjecture, a characterization for rationally connected varieties  has been proven recently, under some stronger conditions implying
in particular that for all $p$ the vector bundles $\Omega^p_X$ are not pseudoeffective (cf. \cite{campanademaillypeternell}):

\begin{theorem}[Campana-Demailly-Peternell]
Let $X$ be a complex projective manifold. Then $X$ is rationally connected 
if and only if for any ample line bundle $A$ on $X$,
for all  $p>0$ and
for all $k>0$ there exists a constant $C_A>0$ such that
\[
H^0(X, \Sym^m \Omega_X^p \otimes A^{\otimes k}) =0
\] for all $m > C_A k$.

\end{theorem}

We remark that rationally connected manifolds are simply connected. 
We have the following result (cf. \cite{brunebarbecampana}) relating simple connection and symmetric tensors, in the direction of Mumford's conjecture above:

\begin{theorem}[Brunebarbe-Campana] 
Let $X$ be a  compact K\"ahler manifold. 
Suppose that $H^0(X, \Sym^m \Omega_X^p) =0$ for all $m>0$ and all $p>0$. Then $X$ is simply connected. Furthermore under the conditions above $X$ is projective.
\end{theorem}

Restricting the conditions above to the case $p=1$, we have the following conjecture
attributed to Esnault:

\begin{conjecture}
\label{conjesnault}
Let $X$ be a compact K\"ahler manifold.
If 
\[
H^0(X, \Sym^m \Omega_X^1) =0 \textrm{ for all $m>0$}
\]
then the fundamental group of $X$ is 
finite.

\end{conjecture}

It has been proven the (slightly) weaker statement that in this case the
fundamental group of $X$ admits no linear representation with infinite image
(cf. \cite{brunebarbe}):

\begin{theorem}[Brunebarbe-Klinger-Totaro]

Let $X$ be a compact K\"ahler manifold. Suppose that there is a finite
dimensional representation of $\pi_1 (X)$, over some field, with infinite image. 
Then $X$ has
a nonzero holomorphic symmetric differential.

\end{theorem}

Conversely, one could wonder whether the presence of one (or many) non vanishing
symmetric differential implies that the fundamental group is infinite.

\begin{question}
\label{questfundgroup}

Let $X$ be a compact complex manifold, suppose that there exists $m>0$
such that 
\(
H^0(X, \Sym^m \Omega_X^1) \neq 0 ~.
\)
Is the fundamental group of $X$ infinite?

\end{question}

If $X$ is K\"ahler and $m=1$ then the answer is positive beacause of Hodge decomposition,
but in general it fails for higher $m$, even restricting to projective manifolds. In fact
 there are varieties with ample cotangent bundle (hence a lot of symmetric differential forms) that are simply connected:
a recent result of Brotbek and Darondeau (cf. \cite{brotbek}) shows that the  cotangent bundle of a general complete intersection
in $\PP^N$ of high degree and dimension $n \leqslant N/2$ is ample.
If $n \geqslant 2$ Lefschetz theorem implies that such a complete intersection is simply connected. Therefore we cannot hope to have a converse of 
Conjecture \ref{conjesnault} and a positive answer in general to Question
\ref{questfundgroup}. The counterexamples however is a projective variety
 of general type (as in particular $\omega_X$ is ample). Therefore we can state the following
 
 \begin{conjecture}
\label{azzardata}
 
 Let $X$ be a compact complex manifold. Suppose 
 that $\dim X < k(X)$  and $\Omega^1_X$ is weakly semiample.
 Then the fundamental group of $X$ is infinite.

 \end{conjecture}

A weakly semiample vector bundle is nef. In the K\"ahler case 
the conjecture above is a consequence  (cf. Remark \ref{remarkclaudon} below)  of the following conjecture by  Wu and Zheng 
(cf. \cite{wuzheng}):

\begin{conjecture}[Wu-Zheng]
\label{2cinesi}

Let $X$ be a compact K\"ahler manifold, such that $\Omega^1_X$ is nef. Then 
a finite étale cover $X^{\prime}$ of $X$ admits a smooth fibration in complex tori 
$X^{\prime} \to Y$  onto a 
 manifold $Y$
of dimension $\dim Y = k(X)$ and
such that the canonical bundle $\omega_Y$ is ample.

\end{conjecture}

\begin{remark}
\label{remarkclaudon}

A  result of
Claudon (cf. \cite{claudon}) shows that if $X$ is a compact K\"aher manifold,
and $f \colon X \to Y$ is a smooth fibration 
in complex tori onto a compact K\"ahler manifold $Y$, then the fundamental group of a fiber $F$
injects into the fundamental group of the total space $X$.
Therefore we get that 
for a fibration in tori $X^{\prime} \to Y$ as above, the fundamental group of $X^{\prime}$ is infinite as soon as the dimension of $Y$ is smaller than that of $X$ (complex tori have infinite fundamental groups), therefore 
Conjecture \ref{2cinesi} implies Conjecture \ref{azzardata} in the K\"ahler case.
Claudon's result does not hold for compact complex manifolds in general:
Hopf surfaces give an example where the map from the fundamental groups of the elliptic fiber to the 
surface is not injective, however the cotangent bundle of a Hopf surface is not semiample, and the Hopf surface has an infinite fundamental group in any case.

\end{remark}

A recent theorem of H\"oring confirms Conjecture \ref{2cinesi} in the projective case,
 assuming
 the canonical  line bundle $\omega_X$ is semiample:

\begin{theorem}[H\"oring, cf. \cite{hoering}]
\label{hoering}

Let $X$ be a smooth projective variety. Suppose that $\Omega^1_X$ is nef and 
that $\omega_X$ is semiample. Then Conjecture \ref{2cinesi} holds for $X$.

\end{theorem}

Other interesting results in the compact complex case, considering the relations between the existence of particular holomorphic symmetric differentials and infiniteness of the fundamental group, are the following theorems:

\begin{theorem}[Bogomolov-De Oliveira, cf. \cite{bogomolov}]

Let $X$ be a compact complex manifold.
If there exists a nontrivial locally exact holomorphic symmetric differential of rank $1$ on $X$, then the fundamental group of $X$ is infinite.

\end{theorem}

\begin{theorem}[Biswas-Dumitrescu, cf. \cite{biswasdumitrescu}]

Let $X$ be a compact complex manifold.
If there exists a nowhere degenerate holomorphic symmetric differential of degree $2$ on $X$, 
 \emph{i.e.} a form $\omega \in H^0(X , \Sym^2 \Omega^1_X)$ such that $\omega(x)$ is a non-degenerate 
quadratic form on $T_{X,x}$,
then the fundamental group of $X$ is infinite.

\end{theorem}

We can confirm some cases of Conjecture \ref{azzardata} with the  following
theorem, which generalise a result in \cite{mistrettaparma}:

\begin{theorem}
\label{fundgroup}

Let $X$ be a compact complex manifold of dimension $n$ and  Kodaira dimension
$k(X)$.

\begin{enumerate}

\item If $k(X) = 0$ and $\Omega^1_X$ is weakly semiample,
then the fundamental group of $X$ is infinite.

\item If $k(X) <n$, $X$ is projective, and $\Omega^1_X$ is weakly semiample,
then the fundamental group of $X$ is infinite.

\end{enumerate}

\end{theorem}

\begin{proof}

Case (i) follows from Theorem \ref{hyperell2}, 
as parallelizable manifolds have infinite fundamental group (cf. Remark
\ref{remfundparall} below).

Case (ii) follows from H\"oring's Theorem \ref{hoering} as explained above, and it can be proven explicitly  as follows: 
it is proven in \cite{hoering} that a smooth projective variety 
$X$ with weakly semiample cotangent bundle admits a finite étale cover 
$X^{\prime} \to X$ 
which is a product $ X^{\prime} = A \times Y$ of an abelian variety $A$ of dimension $n - k(X)$ and 
 a smooth projective manifold $Y$ of dimension $\dim Y = k(X)$ with ample canonical line bundle $\omega_Y$.
 Then as soon as $k(X) =  \dim (Y) < \dim X$ the fundamental group of 
 $X^{\prime}$ (so the one  of $X$) is infinite.

\end{proof}

\begin{remark}
\label{remfundparall}

Let $P$ be a compact complex parallelizable manifold. Then the fundamental group
$\pi_1 (X)$ is infinite.
In fact
suppose $P = H / \Gamma$ 
with $H$ a complex Lie
group and $\Gamma$ a discrete subgroup.
If $\Gamma$ is finite then $H$ is a compact complex Lie group, so it is a 
complex torus, which has an infinite fundamental group contained in the fundamental group of $P$.
If $\Gamma$ is infinite then $\pi_1(P)$ is infinite as well, as it contains $\Gamma$.

\end{remark}

\subsection{Bimeromorphic characterization}

We can wonder whether similar arguments can be used to characterize bimeromorphically 
abelian varieties, or complex tori, or compact complex parallelizable manifolds.

In fact a birational characterization of abelian varieties was obtained under 
supposing that the variety admits a good minimal model.

\begin{definition}

Let $X$ be a smooth projective variety over $\CC$. We say that $X$ admits a 
\emph{good minimal model} if there exists a normal projective variety $Y$ with terminal singularities such 
that $X$ is birational to $Y$ and that $mK_Y$ is a base point free divisor  for some integer 
$m>0$.

\end{definition}

In the work \cite{mistrettaav} the second named author obtains the following result:

\begin{theorem}

Let $X$ be a smooth projective complex variety of Kodaira dimension $k(X)$ that admits a good minimal model.
Then $X$ is birational to an Abelian Variety if and only if 
$k(X) = 0 $ and $\Omega^1_X$ is Asymptotically Generically Generated.

\end{theorem}

As under the usual conjecture of Minimal Model Program every smooth projective variety admits a good minimal model, it would be interesting to obtain the result above without making use of that hypothesis.
In particular we could try to obtain a more general bimeromorphic characterization for 
compact complex parallelizable manifolds among all compact complex manifolds, or for complex tori among K\"ahler manifolds.

\begin{question}

Let $X$ be a compact complex manifold of Kodaira dimension $k(X)=0$ such that 
$\Omega^1_X$ is Asymptotically Generically Generated. Is $X$ bimeromorphic to 
a compact complex parallelizable manifold?

\end{question}

\begin{remark}

If $X$ is a compact K\"ahler manifold of dimension at most $3$ then the Minimal Model Program can be applied to $X$, \emph{i.e.} a good minimal model for $X$ exists.
Therefore it should be possible to answer the question above in that case.

In any case it would be interesting to obtain a bimeromorpphic characterization
without making use of the Minimal Model Program, at least in the K\"ahler case.
\end{remark}

%%%% BIBLIOGRAFIAAAA %%%

 \bibliographystyle{amsalpha22}
 \bibliography{parall}

 \noindent
\textsc{Francesco Esposito, Ernesto C. Mistretta\\
Dipartimento di Matematica \\Universit\`a 
degli Studi di Padova\\
via Trieste 63, 35121 Padova, Italy}

\textit{email:}\texttt{ esposito@math.unipd.it, ernesto.mistretta@unipd.it}

\end{document}